\documentclass{article}
\usepackage[utf8]{inputenc}
\usepackage{amsmath,amssymb,amsthm}

\newtheorem{lem}{Lemma}
\newtheorem{prop}{Proposition}
\newtheorem{property}{Property}
\newtheorem{cor}{Corollary}

\newtheorem{thm}{Theorem}

\title{Recovery of Paley-Wiener functions using scattered translates of regular interpolators}
\author{Jeff Ledford}
\date{May 2012}

\begin{document}

\maketitle

\section{Introduction}

By exploiting properties of the Fourier transform, Lyubarskii and Madych (see \cite{paper 3}) were able to recover Paley-Wiener functions from their samples on a Complete Interpolating Sequence $\{x_j\}$ by using tempered splines with knot sequence $\{x_j\}$ whose degree increases to infinity.  Later, Schlumprecht and Sivakumar (see \cite{siva}), were able to prove a similar result by using parametrized scattered translates of the Gaussian rather than tempered splines.  A natural question to ask is that of whether one may use other families of interpolants to produce similar recovery results.  We provide a general theorem on when this is possible, in particular, we show that the specific example of the Poisson kernel $1/(1+x^2)$, parametrized appropriately, may also be used to recover Paley-Wiener functions from their samples on a Complete Interpolating Sequence.

\section{Definitions and Basic Facts}

We adopt the following convention for the Fourier transform of $g\in L^1(\mathbb{R})$,
\begin{equation}\label{FTdef} \hat{g}(\xi):=(2\pi)^{-1/2}\int_{\mathbb{R}}g(x)e^{-ix\xi}dx.
\end{equation} 
When necessary, this definition is extended to distributions in the usual way.  We denote by $PW_\pi$ the following set of functions
\[
PW_\pi=\left\{f\in L^2(\mathbb{R}): \text{supp}(\hat{f})\subset [-\pi,\pi] \right\}.
\]
A member of this set is called a \emph{Paley-Wiener function}.  We call a sequence $\mathcal{X}=\{x_j\}_{j\in\mathbb{Z}}\subset\mathbb{R}$ a \emph{complete interpolating sequence} (CIS) if the corresponding sequence of exponentials $\mathcal{E}=\{e^{-ix_j\xi}\}$ is a Riesz Basis for $L^2([-\pi,\pi])$.  Since our calculations will require it, we review the definition of a Riesz basis, tailored to our situation.  A set of functions $\mathcal{E}=\{e_j\}$ is said to be a Riesz basis for $L^2([-\pi,\pi])$ if the linear span of $\mathcal{E}$ is dense in $L^2([-\pi,\pi]$ and there is a $C>0$ such that 
\begin{equation}\label{RBineq}
\displaystyle C^{-1}\left\| \{a_j\} \right\|_{l^2(\mathbb{Z})}\leq \left\| \sum_{j\in\mathbb{Z}}a_je_j \right\|_{L^2([-\pi,\pi])} \leq C \| \{a_j\} \|_{l^2(\mathbb{Z})}
\end{equation}
for all $\{a_j\}_{j\in\mathbb{Z}}\in l^2(\mathbb{Z})$.\\
We have, for any $g\in L^2([-\pi,\pi])$, the representation 
\[
g(\xi)=\sum_{j\in\mathbb{Z}}a_je^{-ix_j\xi}
\]
for the appropriate $\{a_j\}_{j\in\mathbb{Z}}\in l^2(\mathbb{Z})$, and define the \emph{prolongation operator} $A:L^2([-\pi,\pi])\to L^2([-\pi,\pi])$ by
\begin{equation}\label{ProOp}
Ag(\xi)=A\left(\sum_{j\in\mathbb{Z}}a_je^{-ix_j\xi}\right)=\sum_{j\in\mathbb{Z}}a_je^{-2\pi i x_j}e^{-ix_j\xi},\hspace{.25in}|\xi|\leq\pi.
\end{equation}
Similarly, we can define $A^k$, for any integer $k$, by
\[
A^kg(\xi)=A^k\left(\sum_{j\in\mathbb{Z}}a_je^{-ix_j\xi}\right)=\sum_{j\in\mathbb{Z}}a_je^{-2\pi k i x_j}e^{-ix_j\xi},\hspace{.25in}|\xi|\leq\pi.
\]
In light of \eqref{RBineq}, we have 
\begin{equation}\label{OPbound}
\| A^kg \|_{L^2([-\pi,\pi])}\leq C^2 \| g\|_{L^2([-\pi,\pi])}.
\end{equation}
A similar bound holds for the adjoint $A^{* k}$.  Allowing $\xi\in\mathbb{R}$ in the right hand side of \eqref{ProOp}, we also see that $Ag(\xi)\in L^2_{loc}(\mathbb{R})$.
We are in position to prove the following lemma.
\begin{lem}\label{sq sum}
Suppose that $f\in PW_\pi$ and $\{x_j\}$ is a CIS, then $\{f(x_j)\}\in l^2(\mathbb{Z})$.
\end{lem}
\begin{proof}
Let $\{e_j\}=\{e^{-ix_j\xi}\}$ be the Riesz basis determined by the CIS.  Then the associated ``dual" Riesz basis $\{\tilde{e}_j\}$, (see \cite{young}), satisfies the following condition:
\[
f(\xi) = \sum_{j\in\mathbb{Z}}\langle f,\tilde{e}_j\rangle e^{-ix_j\xi}, \hspace{.25in}\text{ for }\hspace{.25in}|\xi|\leq\pi.
\]
The inner product is the standard inner product on $L^2([-\pi,\pi])$ and the equality is understood in the $L^2([-\pi,\pi])$ sense.  We have that $ f(x_j)=\langle \hat{f},e_j \rangle $ by the inversion formula, thus  $\hat{f}(\xi)=\sum_{j\in\mathbb{Z}} \langle \hat{f},e_j \rangle \tilde{e_j}$.
We use \eqref{RBineq} to get the following bound:
\[
\sum_{j\in\mathbb{Z}}|f(x_j)|^2 \leq {\tilde{C}^2} \| \hat{f}  \|^2_{L^2([-\pi,\pi])}={\tilde{C}^2} \| f \|^2_{L^2(\mathbb{R})}.
\]
\end{proof}

\section{Interpolation Results}
Throughout this section and the rest of the paper we consider a fixed but otherwise arbitrary $f\in PW_\pi$ and CIS $\{x_j\}$.  We exhibit sufficient conditions on a real valued function $\phi(x)$, which we call an \emph{interpolator}, such that the following property holds:
\begin{property}\label{interp}
There is a unique sequence $\{a_j\}_{j\in\mathbb{Z}}\in l^2(\mathbb{Z})$ for which the interpolant
\[
If(x)=\sum_{j\in\mathbb{Z}}a_j\phi(x-x_j)
\]
is continuous and satisfies $If(x_k)=f(x_k)$ for all $k\in\mathbb{Z}$.
\end{property}
\noindent Let $\phi(x)$ satisfy the following assumptions:
\begin{enumerate}
\item[(A1)] $\phi(x),\hat\phi(\xi)\in L^1(\mathbb{R})$.
\item[(A2)] $\hat{\phi}(\xi)\geq 0$ and $\hat{\phi}(\xi)\geq\delta>0$ on $[-\pi,\pi]$.
\item[(A3)] Let $\displaystyle M_j=\sup_{|\xi|\leq\pi}\hat{\phi}(\xi+2\pi j)$, then $M_j\in l^1(\mathbb{Z})$.
\end{enumerate}
Under these assumptions, we will show that property \ref{interp} holds.  We begin with the following lemma.  In the calculations that follow we will combine all constants into a single one, denoted $C$, whose exact value depends on its occurrence but is otherwise irrelevant.

\begin{lem}\label{interpolation proposition}
If $\phi(x)$ satisfies (A1)-(A3), and $\mathbf{A}:l^2(\mathbb{Z})\to l^2(\mathbb{Z})$ is defined by the infinite matrix $\mathbf{A}:=\left( \phi(x_k- x_j)  \right)_{k,j\in\mathbb{Z}}$, then $\mathbf{A}$ is invertible.
\end{lem}
\begin{proof}
We will show that $\mathbf{A}$ is positive definite and bounded, hence invertible.  In particular, we show the following two inequalities.
\begin{equation}\label{pos def 1}
m\sum_{j\in\mathbb{Z}}|a_j|^2\leq \sum_{j,k\in\mathbb{Z}}a_j\bar{a}_k\phi(x_k-x_j)
\end{equation}
\begin{equation}\label{pos def 2}
 \sum_{j,k\in\mathbb{Z}}a_j\bar{a}_k\phi(x_k-x_j) \leq M \sum_{j\in\mathbb{Z}}|a_j|^2
\end{equation}
Both $m$ and $M$ are positive constants independent of $\{a_j\}_{j\in\mathbb{Z}}$.
We first prove \eqref{pos def 1}. 
\begin{align*}
& \sum_{k\in\mathbb{Z}}\sum_{j\in\mathbb{Z}}\bar{a}_k a_j \phi(x_k-x_j) = \dfrac{1}{\sqrt{2\pi}}\sum_{k\in\mathbb{Z}}\sum_{j\in\mathbb{Z}}\bar{a}_k a_j\int_{\mathbb{R}}\hat{\phi}(\xi)e^{i(x_k-x_j)\xi}d\xi \\
&=\dfrac{1}{\sqrt{2\pi}}\int_{\mathbb{R}}\left|\sum_{j\in\mathbb{Z}}a_j e^{-ix_j\xi}   \right|^2 \hat{\phi}(\xi)d\xi \geq \dfrac{1}{\sqrt{2\pi}}\int_{-\pi}^{\pi} \left|\sum_{j\in\mathbb{Z}}a_j e^{-ix_j\xi}   \right|^2 \hat{\phi}(\xi)d\xi \\
&\geq \dfrac{\delta}{\sqrt{2\pi} C^2}\sum_{j\in\mathbb{Z}}|a_j|^2 \geq m \sum_{j\in\mathbb{Z}}|a_j|^2
\end{align*}
Here we have used (A1), (A2), and \eqref{RBineq}.  We also used the Fubini-Tonelli theorem to interchange the double sum and the integral.  This will be justified as long as we can show \eqref{pos def 2}, which we do presently.
\begin{align*}
 &\sum_{j,k\in\mathbb{Z}}a_j\bar{a}_k \phi(x_k-x_j) = \dfrac{1}{\sqrt{2\pi}}\sum_{j,k\in\mathbb{Z}}a_j\bar{a}_k \int_{\mathbb{R}}\hat{\phi}(\xi)e^{i(x_k-x_j)\xi}d\xi \\
&=\dfrac{1}{\sqrt{2\pi}}\int_\mathbb{R}\left|\sum_{j\in\mathbb{Z}}a_j e^{-ix_j\xi}\right|^2\hat{\phi}(\xi)d\xi \\
&\leq \dfrac{1}{\sqrt{2\pi}}\sum_{k\in\mathbb{Z}}\sup_{|\xi|\leq\pi}\hat{\phi}(\xi+2\pi k)\int_{-\pi}^{\pi}\left|A^k\sum_{j\in\mathbb{Z}}a_j e^{-ix_j\xi}\right|^2d\xi \\
&\leq  \dfrac{C^4}{\sqrt{2\pi}}  \sum_{k\in\mathbb{Z}}M_k \sum_{j\in\mathbb{Z}}|a_j|^2  \leq    M \sum_{j\in\mathbb{Z}}|a_j|^2 
\end{align*}
We have used (A3), \eqref{OPbound}, and the Fubini-Tonelli theorem to interchange the integral and the sum.  

\end{proof}
By virtue of Lemma \ref{sq sum}, we have the following corollary.

\begin{cor}\label{cor1}
There is a unique sequence $\{a_j\}_{j\in\mathbb{Z}}\in l^2(\mathbb{Z})$ such that $If(x_k)=f(x_k)$ for all $k\in\mathbb{Z}$.
\end{cor}
\begin{prop}\label{L2prop}
If $\phi(x)$ satisfies (A1)-(A3), then Property \ref{interp} is satisfied and $If(x)\in L^2(\mathbb{R})$.
\end{prop}
\begin{proof}
We need only show that $If(x)$ is continuous, as the rest follows from Corollary \ref{cor1}.  It is enough to show that $\widehat{If}(\xi)\in L^1(\mathbb{R})$.

\begin{align*}
&\int_{\mathbb{R}}\left|\widehat{If}(\xi)\right|d\xi =\int_{\mathbb{R}}\left|\hat{\phi}(\xi)\sum_{j\in\mathbb{Z}}a_je^{-ix_j\xi}\right|d\xi \\
=&\sum_{k\in\mathbb{Z}}\int_{-\pi}^{\pi}\left|\hat{\phi}(\xi+2\pi k)\right|\left|A^{k}\left(\sum_{j\in\mathbb{Z}}a_je^{-ix_j\xi}\right)\right|d\xi   \\
\leq &\sum_{k\in\mathbb{Z}}\sup_{|\xi|\leq\pi}\left|\hat{\phi}(\xi+2\pi k)\right|\int_{-\pi}^{\pi}\left|A^k\left(\sum_{j\in\mathbb{Z}}a_je^{-ix_j\xi}\right)\right|d\xi \\
\leq& C\left\| \{a_j\} \right\|_{l^2(\mathbb{Z})}\sum_{k\in\mathbb{Z}}M_k
\end{align*}

Here we have used (A3) and \eqref{OPbound}, as well as the Cauchy-Schwarz inequality.  To show that $If(x)\in L^2(\mathbb{R})$, we use Plancherel's theorem.
\begin{align*}
&\left\| \sum_{j\in\mathbb{Z}}a_j\phi(x-x_j)  \right\|^2_{L^2(\mathbb{R})}=\int_{\mathbb{R}}\left|\hat\phi(\xi) \sum_{j\in\mathbb{Z}}a_je^{-ix_j\xi} \right|^2d\xi  \\
&\leq\sum_{k\in\mathbb{Z}}\sup_{|\xi|\leq\pi}\left|\hat{\phi}(\xi-2\pi k)\right|^2\int_{-\pi}^{\pi}\left|A^k\left(\sum_{j\in\mathbb{Z}}a_je^{-ix_j\xi}\right)\right|^2d\xi \leq C \sum_{k\in\mathbb{Z}} M^2_k\sum_{j\in\mathbb{Z}}|a_j|^2
\end{align*}
We've used (A3), \eqref{RBineq}, and \eqref{OPbound} to arrive at the desired estimate.  Note that (A3) implies that $\{M_j\}_{j\in\mathbb{Z}}\in l^p(\mathbb{Z})$ for all $p\geq 1$.
\end{proof}

\section{Recovery Results}

We consider the one parameter family of interpolators $\{\phi_\alpha(x)\}$, where $\alpha\in A\subset(0,\infty)$.  We will call the family \emph{regular} if it satisfies the following hypotheses.
\begin{itemize}
\item[(H1)] $\phi_{\alpha}(x)$ satisfies (A1)-(A3) for each $\alpha\in A$.
\item[(H2)] $\displaystyle\sum_{j\neq0}M_j(\alpha)\leq C m_\alpha$, where $M_j(\alpha)$ is as in (A3), $\displaystyle m_\alpha = \inf_{|\xi|\leq\pi}\hat{\phi}_{\alpha}(\xi)$, and C is independent of $\alpha$.
\item[(H3)] $\text{For almost every } |\xi|\leq\pi; \displaystyle\lim_{\alpha\to\infty}\dfrac{m_\alpha}{\hat{\phi}_\alpha(\xi)}=0.$
\end{itemize}

The indexing set $A$ may be continuous or discrete to serve our purpose, but in either case we require that $A\subset(0,\infty)$ is unbounded.  We may now take a similar path as the one laid out in \cite{paper 3}, as well as \cite{siva}.  We introduce the notation:
\[
 I_{\alpha}f(x)=\sum_{j\in\mathbb{Z}}a_j\phi_\alpha(x-x_j)\hspace{.2in }\text{where } I_\alpha f(x_n)=f(x_n) \text{ and }\{a_j\}\in l^2(\mathbb{Z}).
\]
In light of Proposition \ref{L2prop}, $I_\alpha f \in L^2(\mathbb{R})$; thus we may use the Fourier transform, which is given by:
\begin{equation}\label{transform}
\widehat{I_\alpha f}(\xi) = \hat{\phi}_\alpha (\xi)\sum_{j\in\mathbb{Z}}a_je^{-ix_j \xi}=\hat{\phi}_\alpha(\xi)\psi_\alpha(\xi).
\end{equation}

\begin{prop} The function $\psi_{\alpha}(\xi)$ satisfies the following relationship:
$$ \hat{\phi}_{\alpha}(\xi)\psi_{\alpha}(\xi) +\sum_{j\neq 0}A^{*j}\left(  \hat{\phi}_{\alpha}(\eta+2\pi j)A^j \psi_{\alpha}(\eta)\right)(\xi)=\hat{f}(\xi)$$
for $-\pi\leq\xi\leq\pi$ where A is the prolongation operator and A* is its adjoint.

\begin{proof}  Write 

\begin{align*}
\int_{-\pi}^{\pi}\hat{f}(\xi)e^{ix_n\xi}d\xi &=\sqrt{2\pi} f(x_n)=\sqrt{2\pi} I_{\alpha}f(x_n)\\
&=\int_{\mathbb{R}}\hat{\phi}_{\alpha}(\xi)\psi_{\alpha}(\xi)e^{ix_n\xi}d\xi\\
&=\sum_{j\in\mathbb{Z}}\int_{(2j-1)\pi}^{(2j+1)\pi}\hat{\phi}_{\alpha}(\xi)\psi_{\alpha}(\xi)e^{ix_n\xi}d\xi\\
&=\sum_{j\in\mathbb{Z}}\int_{-\pi}^{\pi}\hat{\phi}_{\alpha}(\xi+2\pi j)A^j\psi_{\alpha}(\xi)e^{ix_n(\xi+2\pi j)}  d\xi\\
&=\sum_{j\in\mathbb{Z}}\int_{-\pi}^{\pi} A^{*j}\left(\hat{\phi}_{\alpha}(\eta+2\pi j)A^j \psi_{\alpha}(\eta)\right)(\xi)e^{ix_n\xi}  d\xi\\
&=\int_{-\pi}^{\pi}\left\{\sum_{j\in\mathbb{Z}} A^{*j}\left(\hat{\phi}_{\alpha}(\eta+2\pi j)A^j \psi_{\alpha}(\eta)\right)(\xi)\right\}e^{ix_n\xi}  d\xi
\end{align*}

Since $\{x_n\}$ is a CIS and the above equations hold for all $n\in\mathbb{Z}$ we have:

\begin{equation}\label{eq1}
\hat{f}(\xi)=\sum_{j\in\mathbb{Z}}A^{*j}\left( \hat{\phi}_{\alpha}(\eta+2\pi j)A^j \psi_{\alpha}(\eta) \right)(\xi).
\end{equation}

\end{proof}
\end{prop}

We rewrite \eqref{eq1} as

\begin{equation}\label{eq2}
\hat{f}(\xi)=\widehat{I_{\alpha}f}(\xi)+ \sum_{j\neq 0}A^{*j}\left( \dfrac{\hat{\phi}_{\alpha}(\eta+2\pi j)}{m_{\alpha}} A^j \left(m_{\alpha}\psi_{\alpha}(\eta)\right) \right)(\xi)
\end{equation}

We can abbreviate \eqref{eq2} as

\begin{equation}\label{eq3}
\hat{f}(\xi)=\widehat{I_{\alpha}f}(\xi)+ B_{\alpha}\left( m_{\alpha}\psi_{\alpha}(\eta)\right)(\xi)
\end{equation}

where

\begin{equation}\label{B}
B_{\alpha}g(\xi)= \sum_{j\neq 0}A^{*j}\left( \dfrac{\hat{\phi}_{\alpha}(\eta+2\pi j)}{m_{\alpha}} A^j (g(\eta))\right)(\xi).
\end{equation}

In order to proceed, we will need the following

\begin{prop}

\begin{equation}\label{eq4}\left\|  B_{\alpha}g(\xi) \right\|_{L^2([-\pi,\pi])}\leq C\left\|g(\xi)\right\|_{L^2([-\pi,\pi])},\end{equation} where C is independent of both $\alpha$ and g.
\end{prop}

\begin{proof}
First note that the operator norms of both $A^j$ and $A^{*j}$, for $j\in\mathbb{Z}$, are bounded uniformly.
Now,

\begin{align*}
&\left\| \sum_{j\neq 0}A^{*j}\left(  \dfrac{\hat{\phi}_{\alpha}(\eta+2\pi j)}{m_{\alpha}} A^j (g(\eta))\right)(\xi) \right\|_{L^2([-\pi,\pi])}\\
 &\leq\sum_{j\neq 0}\left\| A^{*j}\left( \dfrac{\hat{\phi}_{\alpha}(\eta+2\pi j)}{m_{\alpha}} A^j (g(\eta))\right)(\xi) \right\|_{L^2([-\pi,\pi])}\\
&\leq\sum_{j\neq 0}C\left\| \dfrac{\hat{\phi}_{\alpha}(\xi+2\pi j)}{m_{\alpha}} A^j g(\xi)  \right\|_{L^2([-\pi,\pi])}\\
&\leq \dfrac{C}{m_{\alpha}}\sum_{j\neq 0}\sup_{|\xi|\leq\pi}|\hat{\phi}_{\alpha}(\xi-2\pi j)|\left\| A^j g(\xi)\right\|_{L^2([-\pi,\pi])}\\
&\leq \dfrac{C}{m_{\alpha}} \sum_{j\neq 0}M_j(\alpha)  \left\| g(\xi) \right\|_{L^2([-\pi,\pi])}\\
&\leq C  \left\| g(\xi) \right\|_{L^2([-\pi,\pi])}\\
\end{align*}

We've used \eqref{OPbound}, (A3), and (H2) to obtain the constant C which is independent of both $\alpha$ and $g(\xi)$.  
\end{proof}

\begin{lem}
If $f \in PW_\pi$ and $\psi_{\alpha}(\xi)$ is as in \eqref{eq2}, then \begin{equation}\label{eq5}
\left\| \psi_{\alpha}(\xi)  \right\|_{L^2([-\pi,\pi])}\leq (1/m_\alpha)\left\| \hat{f}\right\|_{L^2([-\pi,\pi])}.
\end{equation}
\end{lem}

\begin{proof}
We use \eqref{eq2} and take the inner product of both sides with $\psi_{\alpha}(\xi)$.  This yields:
\[
\left\langle \hat{f} ,\psi_{\alpha} \right\rangle = \left\langle{\hat{\phi}_{\alpha}\psi_{\alpha}} ,{\psi_{\alpha}} \right\rangle  +\sum_{j\neq 0} \left\langle {\hat{\phi}_{\alpha}(\xi+2\pi j)A^j\psi_{\alpha}} ,{A^j\psi_{\alpha}}   \right\rangle.
\]
Since all of the summands are positive we have:

\begin{align*}
m_{\alpha}\left\|\psi_{\alpha}\right\|^2_{L^2([-\pi,\pi])}&=m_{\alpha} \left\langle \psi_{\alpha}  , \psi_{\alpha}  \right\rangle \leq \left\langle \hat{\phi}_{\alpha}\psi_{\alpha} ,\psi_{\alpha}  \right\rangle\\
 &\leq \left\langle \hat{f} , \psi_{\alpha} \right\rangle \leq \left\|\hat{f}\right\|_{L^2([-\pi,\pi])}\left\|\psi_{\alpha}\right\|_{L^2([-\pi,\pi])},
 \end{align*}

\noindent which is the desired result.
\end{proof}
\noindent Combining these last two results yields the following.

\begin{lem}\label{LEMMA1}
If $f\in PW_\pi$ then:

\begin{equation}\label{eq6}
\left\| \widehat{I_{\alpha}f}(\xi)\right\|_{L^2([-\pi,\pi])}\leq C \left\| \hat{f}(\xi)\right\|_{L^2([-\pi,\pi])}
\end{equation}

\noindent where $C>0$ is a constant independent of both $f$ and $\alpha$. 
\end{lem}

\begin{proof}
In view of \eqref{eq3}, we have
\[
\left\| \widehat{I_{\alpha}f}(\xi)\right\|_{L^2([-\pi,\pi])}\leq \left\|\hat{f}(\xi)\right\|_{L^2([-\pi,\pi])} + \left\|  B_{\alpha}\big( m_{\alpha}\psi_{\alpha}\big)(\xi) \right\|_{L^2([-\pi,\pi])},
\]
and \eqref{eq5} together \eqref{eq4} yield:
\[
\left\| \widehat{I_{\alpha}f}(\xi)\right\|_{L^2([-\pi,\pi])}\leq (1+C)\left\| \hat{f}\right\|_{L^2([-\pi,\pi])} \leq C\left\| \hat{f}\right\|_{L^2([-\pi,\pi])} .
\]
\end{proof}

\begin{prop}
The mapping $I_{\alpha}: PW_\pi \to L^2(\mathbb{R})$, given by $f(x) \mapsto I_{\alpha}f(x)$, is bounded uniformly with respect to $\alpha$.  Which is to say that if $f\in PW_\pi$ then:

\begin{equation}\label{eq7}
\left\| I_{\alpha}f(x)  \right\|_{L^2(\mathbb{R})} \leq C \left\| f(x) \right\|_{L^2(\mathbb{R})},
\end{equation}
where $C>0$ is independent of both $\alpha$ and $f$.
\end{prop} 

\begin{proof}
Plancherel's theorem implies that \eqref{eq7} is equivalent to

\begin{equation}\label{eq8}
\left\| \widehat{I_{\alpha}f}(\xi)   \right\|_{L^2(\mathbb{R})} \leq C \left\| \hat{f}(\xi) \right\|_{L^2([-\pi,\pi])}.
\end{equation}
Now we write:

\begin{align*}
&\left\| \widehat{I_{\alpha}f}(\xi)   \right\|^2_{L^2(\mathbb{R})}=\left\| \widehat{I_{\alpha}f}(\xi)   \right\|^2_{L^2([-\pi,\pi])} + \sum_{j\neq 0}\int_{-\pi}^{\pi}\left|\hat{\phi}_{\alpha}(\xi+2\pi j) A^j \psi_{\alpha}(\xi)  \right|^2d\xi \\
&\leq C^2 \left\| \hat{f}(\xi)\right\|^2_{L^2([-\pi,\pi])} + C^4\sum_{j\neq 0}\left(\sup_{|\xi|\leq\pi}|\hat\phi_{\alpha}(\xi+2\pi j)|\right)^2\left\|\psi_{\alpha}\right\|^2_{L^2([-\pi,\pi])}\\
&\leq C^2 \left\| \hat{f}(\xi)\right\|^2_{L^2([-\pi,\pi])} + \dfrac{C^4}{m_\alpha^2}\left\{\sum_{j\neq 0}M_j(\alpha)\right\}^2\left\| \hat{f}(\xi)\right\|^2_{L^2([-\pi,\pi])}
\end{align*}
Here we have used Lemma \ref{LEMMA1}, \eqref{OPbound}, (H2), and the fact that all the terms in the sum are positive so that $\| \{M_j(\alpha)\} \|_{l^2}^2\leq \| \{M_j(\alpha)\} \|_{l^1}^2$ holds.  Thus, we have shown that for a positive constant C, which is independent of both $\alpha$ and $f$,
\[
\left\| \widehat{I_{\alpha}f}(\xi)   \right\|_{L^2(\mathbb{R})} \leq C \left\|\hat{f}(\xi)\right\|_{L^2([-\pi,\pi])}.
\]
\end{proof}
Now we are in position to prove one of the main results.

\begin{thm}
If $f\in PW_\pi$, then
\begin{equation}\label{eq9}
\lim_{\alpha\to\infty}\left\| f(x) - I_{\alpha} f(x)\right\|_{L^2(\mathbb{R})}=0.
\end{equation}
\end{thm}

\begin{proof}
We use Plancherel's theorem and show that 

\begin{equation}\label{eq10}
\lim_{\alpha\to\infty}\left\| \hat{f}(\xi) - \widehat{I_{\alpha} f}(\xi)\right\|_{L^2(\mathbb{R})}=0.
\end{equation}
Since $f\in PW_\pi$ we may write

\begin{equation}\label{eq11}
\left\| \hat{f}(\xi) - \widehat{I_{\alpha} f}(\xi)\right\|^2_{L^2(\mathbb{R})}= \left\| \hat{f}(\xi) - \widehat{I_{\alpha} f}(\xi)\right\|^2_{L^2([-\pi,\pi])} + \sum_{j\neq 0}\left\| A^j\widehat{I_{\alpha}f}(\xi) \right\|^2_{L^2([-\pi,\pi])}.
\end{equation}
We estimate the two terms on right hand side of the previous equation separately.  For the first term in \eqref{eq11} we rewrite \eqref{eq2} as:

\[
\hat{f}(\xi)=\widehat{I_{\alpha}f}(\xi) + \sum_{j\neq 0}A^{*j}\left( \dfrac{\hat{\phi}_{\alpha}(\eta+2\pi j)}{m_{\alpha} }\left( A^j \left[ \dfrac{m_{\alpha}}{\hat{\phi}_{\alpha}(\omega)} \widehat{I_{\alpha}f}(\omega) \right] \right)(\eta)   \right)(\xi)
\]
for $-\pi\leq\xi\leq\pi.$  We abbreviate this expression to 
\begin{equation}\label{eq12}
\left(I+B_{\alpha}T_{\alpha}\right)\widehat{I_{\alpha}f}(\xi)=\hat{f}(\xi)
\end{equation}
for $|\xi|\leq\pi$.  Notice that in view of \eqref{eq6} and the fact that $f\in PW_\pi$ we have that both $\hat{f}$ and $\widehat{I_{\alpha}f}$ are in $L^2([-\pi,\pi])$ and $I, B_{\alpha},$ and $T_{\alpha}$ are linear operators on $L^2([-\pi,\pi])$, where $I$ is the usual identity operator, $B_{\alpha}$ is defined by  \eqref{B}, and 

\begin{equation}\label{M}
T_{\alpha}g(\xi)=\dfrac{m_{\alpha}}{\hat{\phi}_{\alpha}(\xi)}g(\xi).
\end{equation}
The content of the last proposition is that the operator $(I+B_{\alpha} T_{\alpha})$ is invertible as a mapping from $L^2([-\pi,\pi])$ to itself and that the inverses are uniformly bounded, i.e.

\begin{equation}\label{opnorm}
\left\| (I+B_{\alpha}T_{\alpha})^{-1}\right\|_{op} \leq C.
\end{equation}
So we may write
\[
\hat{f}-\widehat{I_{\alpha}f}=\hat{f}-(I+B_{\alpha}T_{\alpha})^{-1}\hat{f}=(I+B_{\alpha}T_{\alpha})^{-1}(B_{\alpha}T_{\alpha})\hat{f}.
\]
So, in view of \eqref{opnorm} and \eqref{eq4} we have that:
\begin{equation}\label{bound1}
\left\| \hat{f}-\widehat{I_{\alpha}f} \right\|_{L^2([-\pi,\pi])} \leq C\left\|T_{\alpha}\hat{f} \right\|_{L^2([-\pi,\pi])}.
\end{equation}
We move now to the second term in \eqref{eq11}.
\[
\sum_{j\neq 0}\left\| A^j\widehat{I_{\alpha}f}(\xi) \right\|^2_{L^2([-\pi,\pi])}=\sum_{j\neq 0}\int_{-\pi}^{\pi}\left|\dfrac{\hat{\phi}_{\alpha}(\xi+2\pi j)}{m_{\alpha}} A^j T_{\alpha}\widehat{I_{\alpha}f}(\xi) \right|^2d\xi
\]
Estimates completely analogous to those used in \eqref{eq8} show that
\begin{align}\label{bound2}
&\sum_{j\neq 0} \left\| A^j  \widehat{I_{\alpha}f}(\xi) \right\|^2_{L^2([-\pi,\pi])} \leq C \sum_{j\neq 0}\dfrac{M_j(\alpha)}{m_\alpha}\left\| T_\alpha \widehat{I_\alpha f}(\xi)  \right\|_{L^2([-\pi,\pi])}^{2}\\
\nonumber& \leq C\left\| T_{\alpha}\widehat{I_{\alpha}f} \right\|_{L^2([-\pi,\pi])}^2 ,
\end{align}
where $C>0$ is independent of $f$ and $\alpha$.  Now we have
\begin{eqnarray*}
\left\| T_{\alpha}\widehat{I_{\alpha}f}  \right\|_{L^2([-\pi,\pi])}   \leq & \left\| T_{\alpha}\hat{f}  \right\|_{L^2([-\pi,\pi])} + \left\| T_{\alpha}(\hat{f}-\widehat{I_{\alpha}f})  \right\|_{L^2([-\pi,\pi])}  \\
\leq &\left\| T_{\alpha}\hat{f}  \right\|_{L^2([-\pi,\pi])}  + \left\| \hat{f}-\widehat{I_{\alpha}f}  \right\|_{L^2([-\pi,\pi])}
\end{eqnarray*}
the last inequality follows from the fact that when $|\xi|\leq\pi$ we have $$ \dfrac{m_{\alpha}}{\hat{\phi}_{\alpha}(\xi)} \leq \dfrac{m_{\alpha}}{m_{\alpha}} \leq1 .$$
Finally, from \eqref{bound1} we see that \eqref{bound2} may be written as:
\begin{equation}\label{bound3}
\sum_{j\neq 0}\left\| A^j\widehat{I_{\alpha}f}(\xi) \right\|^2_{L^2([-\pi,\pi])} \leq C \left\| T_{\alpha}\hat{f} \right\|^2_{L^2([-\pi,\pi])}.
\end{equation}
where $C$ is independent of both $f$ and $\alpha$.

\noindent Combining \eqref{eq11} with the estimates \eqref{bound1} and \eqref{bound3} yields a positive constant $C_1$, independent of $f$ and $\alpha$, such that

\begin{equation}\label{bound}
\left\| \hat{f}(\xi) - \widehat{I_{\alpha} f}(\xi)\right\|^2_{L^2(\mathbb{R})} \leq C_1\left\| T_{\alpha}\hat{f}  \right\|^2_{L^2([-\pi,\pi])} 
\end{equation}
Using (H3), we see that an application of the dominated convergence theorem yields the conclusion of the theorem.
\end{proof}
\noindent This result yields another on the pointwise convergence of $I_{\alpha}f(x)$.  We have the following
\begin{thm}
If $f\in PW_\pi$, then for all $x\in\mathbb{R}$

\begin{equation}\label{thm2}
\big| f(x) - I_{\alpha}f(x)   \big| \leq C \left\| T_{\alpha}\hat{f}(\xi)   \right\|_{L^2([-\pi,\pi])},
\end{equation}
where $C$ is a constant independent of both $f$ and $\alpha$.  Thus we have that
\[ \lim_{\alpha\to\infty} I_{\alpha}f(x)=f(x)
\]
uniformly on $\mathbb{R}$.
\end{thm}
\begin{proof}
We will use the Cauchy-Schwarz inequality.
\begin{align*}
\left| f(x)-I_{\alpha}f(x)  \right|&= \dfrac{1}{\sqrt{2\pi}}\left|\int_{-\pi}^{\pi} (\hat{f}(\xi)- \widehat{I_{\alpha}f}(\xi))e^{ix\xi}d\xi  +  \sum_{j\neq 0}     \int_{-\pi}^{\pi}e^{ix\xi}A^j \widehat{I_{\alpha}f} (\xi)d\xi \right|\\
&\leq \dfrac{1}{\sqrt{2\pi}}   \left\{\int_{-\pi}^{\pi}\left| \hat{f}(\xi)- \widehat{I_{\alpha}f}(\xi)  \right|d\xi  +\sum_{j\neq 0}\int_{-\pi}^{\pi}\left| A^j\widehat{I_{\alpha}f}(\xi) \right|d\xi  \right\} \\
&\leq   \left\| \hat{f}(\xi)- \widehat{I_{\alpha}f}(\xi) \right\|_{L^2([-\pi,\pi])}    + \sum_{j\neq0 }\left\| A^j \widehat{I_{\alpha}f} (\xi)   \right\|_{L^2([-\pi,\pi])}  \\
& \leq (C_1+C_2) \left\| T_{\alpha}\hat{f}(\xi) \right\|_{L^2([-\pi,\pi])}=C \left\| T_{\alpha}\hat{f}(\xi) \right\|_{L^2([-\pi,\pi])}
\end{align*}
where the last inequality comes from \eqref{bound1} and \eqref{bound3}.
This is the desired result since $C$ is independent of both $\alpha$ and $f$.  As $\alpha$ increases without bound, we get the desired pointwise limit.  Notice that this convergence does not depend on $x$, hence the convergence is uniform in $\mathbb{R}$.
\end{proof}

\section{Examples}
\noindent This section deals with two examples which have not appeared in the literature.  The first example is the family of Poisson kernels $\mathcal{F}=\{\sqrt{2/\pi}\alpha(\alpha^2+x^2)^{-1}\}_{\alpha\geq 1}$, and the second example deals with the forward difference of the multiquadric $\sqrt{1+x^2}$.

\begin{prop}
$\mathcal{F}=\{\sqrt{2/\pi}\alpha(\alpha^2+x^2)^{-1}\}_{\alpha\geq 1}$ is a regular family of interpolators.
\end{prop}
\begin{proof}
Letting $\phi_\alpha(x)=\sqrt{\dfrac{2}{\pi}}\dfrac{\alpha}{\alpha^2+x^2}$, we find that $\hat{\phi}_\alpha(\xi)=e^{-\alpha|\xi|}$, $m_\alpha=e^{-\alpha\pi}$, and $M_j(\alpha)=e^{-\alpha\pi(2|j|-1)}$ for $j\neq 0$ and $M_0(\alpha)=1$.  Now checking the required hypotheses is a routine exercise.
\end{proof}

The next example is more involved and requires the introduction of the following notation.  We let $\Delta^1g(x)=g(x+1)+g(x-1)-2g(x)$ and for $k=2,3,\dots$ we set $\Delta^kg(x)=\Delta^1(\Delta^{k-1}g)(x)$.  We write $f_{*}^{k}(x)$ for the $k$-fold convolution of $f(x)$, that is 
\[
f_{*}^{k}(x)=\underbrace{f*\cdots*f}_{k \text{ times}}(x).
\]
\begin{prop}\label{p1}
$\mathcal{G}=\{(-1)^k\Delta^k\phi_{*}^{k}(x)\}_{k= 1}^{\infty}$, where $\phi(x)=\sqrt{1+x^2}$, is a regular family of interpolators.
\end{prop}

We begin with the Fourier transform.
\[
\widehat{(-1)^k\Delta^k\phi_{*}^{k}(x)}=C_k \left[\dfrac{2(1-\cos(\xi))}{\xi^2}\right]^k[|\xi|K_1(|\xi|)]^k,
\]
where $K_1$ is the Macdonald function (see \cite{wendland}) and the value of $C_k$ may be found in \cite{wendland} as well, but it does not affect the calculations that follow.  Thus, we will subsequently omit $C_k$ in what follows.  To simplify notation a bit, we let
\begin{equation}
\hat{u}_k(\xi)= \left[\dfrac{2(1-\cos(\xi))}{\xi^2}\right]^k[|\xi|K_1(|\xi|)]^k.
\end{equation}  
The proof of the proposition rests on the following lemmas.

\begin{lem}\label{l1}
$\hat{u}^{(j)}_k(\xi)\in L^1(\mathbb{R})$ for $j=0,1,2$.
\end{lem}
\begin{proof}
Let $f(\xi)=\dfrac{2(1-\cos(\xi))}{\xi^2}$ and $g(\xi)= |\xi|K_1(|\xi|)$, then we have:
\begin{align*}
\hat{u}_k(\xi) & = [f(\xi)g(\xi)]^k = O(1),|\xi|\to 0,\\
\hat{u}'_{k}(\xi) & = k[f(\xi)g(\xi)]^{k-1}[f(\xi)g'(\xi)+f'(\xi)g(\xi)] = O(|\xi|\ln(|\xi|)),|\xi|\to 0,\\
\hat{u}''_{k}(\xi) & = k(k-1)[f(\xi)g(\xi)]^{k-2}\left[ \dfrac{d}{d\xi}\left( f(\xi)g(\xi) \right)\right] ^2\\
&+k[f(\xi)g(\xi)]^{k-1}\dfrac{d^2}{d\xi^2}\left( f(\xi)g(\xi) \right)=O(\ln(|\xi|)),|\xi|\to 0.
\end{align*}
To get the asymptotic estimates, we expanded $f(\xi)$ in a Taylor series centered at $\xi=0$, and used the estimates found in \cite{AS} for $g(\xi)$.  Noting that each of $\hat{u}_k(\xi),\hat{u}'_k(\xi),$ and $\hat{u}''_k(\xi)$ have exponential decay as $|\xi|\to\infty$, see \cite{AS}, we see that each function is integrable.
\end{proof}
\begin{proof}(of Proposition \ref{p1})

As a consequence of Lemma \ref{l1}, we have that both $u_k(x)\in L^1(\mathbb{R})$ and $\hat{u}_{k}(\xi)\in L^1(\mathbb{R})$, hence $(A1)$ is satisfied for each $k=1,2,3,\dots$  As for $(A2)$, it is clear that $\hat{u}_{k}(\xi)\geq 0$, and examining the derivative we see that it is negative on $(0,2\pi)$, hence $\hat{u}_{k}(\xi)\geq \hat{u}_{k}(\pi) > 0$.  For the bounding sequence in $(A3)$, we again apply estimates found in \cite{wendland}, for $r>0$:
\[
\sqrt{\pi/2}r^{-1/2}e^{-r}\leq K_{\beta}(r) \leq \sqrt{2\pi}r^{-1/2}e^{-r}e^{-\beta^2/(2r)},
\]
and use 
\[
M_j(k)=\left\{
        \begin{array}{ll}
            2^{3k/2}\pi^{-k}(2|j|-1)^{-3k/2}e^{-(2|j|-1)k\pi}; & \quad |j|>1 \\
            \hat{u}_k(\pi); & \quad |j| = 1 \\
            1; & \quad j=0.
        \end{array}
    \right.
\]
Thus, $M_j(k)\in l^1(\mathbb{Z})$ for all $k=1,2,3,\dots$, hence $(H1)$ is satisfied.  We also use this estimate above to get $m_k \geq 2^{3k/2}\pi^{-k}e^{-k\pi}$.  This allows us to check $(H2)$ and $(H3)$.  We have
\[
\sum_{j\neq 0}M_j(k)/m_k \leq 2+ (2/3)^k\sum_{|j|>1}e^{-2\pi(|j|-1)}\leq 2+\dfrac{4/3}{e^{2\pi}-1},
\]
so $(H2)$ is satisfied.  For $(H3)$, we note that since $\hat{u}_k$ decreases, $m_k < \hat{u}_k(\xi)$ if $|\xi|<\pi$. Now the  $k$th power in the transforms forces the limit to $0$.  Thus $\mathcal{G}$ is a regular family of interpolators.
\end{proof}

\end{document}